\documentclass[12pt]{smfart}
\usepackage{color}
\usepackage{amssymb,verbatim}
\usepackage{amsthm,array,amssymb,amscd,amsfonts,amssymb,latexsym, url}
\usepackage{amsmath}
\usepackage[all]{xy}
\usepackage[french]{babel}

\newcounter{spec}
{\end{list}}

\renewcommand{\P}{{\mathbf P}}

\newcommand{\C}{{\mathbb C}}

\renewcommand{\L}{{\mathbb L}}

\renewcommand{\lim}{\varprojlim}

\numberwithin{equation}{section}

\newfont{\gothic}{eufb10}

\renewcommand{\qed}{{\hfill$\square$}}

\newtheorem{theo}{Th\'{e}or\`{e}me}[section]
\newtheorem{prop}[theo]{Proposition}

\newtheorem{lem}[theo]{Lemme}
\newtheorem{cor}[theo]{Corollaire}
\theoremstyle{definition}
\newtheorem{defi}[theo]{D\'efinition}
\theoremstyle{remark}
\newtheorem{rema}[theo]{Remarque}

\newcommand{\bthe}{\begin{theo}}
\newcommand{\ble}{\begin{lem}}
\newcommand{\bpr}{\begin{prop}}
\newcommand{\bco}{\begin{cor}}
\newcommand{\bde}{\begin{defi}}
\newcommand{\ethe}{\end{theo}}
\newcommand{\ele}{\end{lem}}
\newcommand{\epr}{\end{prop}}
\newcommand{\eco}{\end{cor}}
\newcommand{\ede}{\end{defi}}

\newcommand{\F}{{\mathbb F}}

\DeclareFontFamily{U}{wncy}{}
\DeclareFontShape{U}{wncy}{m}{n}{%
<5>wncyr5%
<6>wncyr6%
<7>wncyr7%
<8>wncyr8%
<9>wncyr9%
<10>wncyr10%
<11>wncyr10%
<12>wncyr6%
<14>wncyr7%
<17>wncyr8%
<20>wncyr10%
<25>wncyr10}{}
\DeclareMathAlphabet{\cyr}{U}{wncy}{m}{n}

\begin{document}

\title[Surfaces de del Pezzo de degr\'e 4 ]{Surfaces de del Pezzo de degr\'e 4 sur un corps $C_{1}$}
\author{Jean-Louis Colliot-Th\'el\`ene}
\address{C.N.R.S., Universit\'e Paris Sud\\Math\'ematiques, B\^atiment 425\\91405 Orsay Cedex\\France}
\email{jlct@math.u-psud.fr}
 
\date{soumis le 5 avril  2014, r\'evis\'e le 2 ao\^{u}t  2015;  \`a para\^{\i}tre dans Taiwanese Journal of Mathematics}
\maketitle

\begin{abstract}  
Sur toute surface de del Pezzo de degr\'e 4 sur un corps $C_{1}$  de caract\'eristique z\'ero,
tous les points rationnels sont $R$-\'equivalents. Plus g\'en\'eralement, ceci vaut sur tout
corps parfait infini de caract\'eristique diff\'erente de 2.
\end{abstract}

\begin{altabstract}
On a del Pezzo surface of degree 4 over an infinite  $C_{1}$-field of characteristic
zero,  all rational points are $R$-equivalent. This more generally holds over any
infinite perfect $C_{1}$-field of characteristic different from 2.
 \end{altabstract}

\section{Introduction}

Soient $k$ un corps,  $X$ une $k$-vari\'et\'e alg\'ebrique et $X(k)$
l'ensemble de ses points rationnels.
 On
dit que deux points   $P$ et $Q$ de $X(k)$ sont directement $R$-li\'es
  s'il existe un ouvert $D $ de la droite projective $\P^1_{k}$
et un $k$-morphisme $D \to X$ tels que $P$ et $Q$ appartiennent \`a l'image de
$D(k)$. La $R$-\'equivalence sur $X(k)$ est la relation d'\'equivalence engendr\'ee
par cette relation \'el\'ementaire. On note $X(k)/R$ l'ensemble des classes d'\'equivalence.

Pour $k=\F$ un corps fini et $X$ une hypersurface lisse de degr\'e $d \leq n$ dans  $\P^n_{\F}$,
 un cas particulier  d'une conjecture de Koll\'ar (voir \cite{K})  affirme que l'ensemble $X(\F)/R$ est r\'eduit \`a un \'el\'ement.
Pour $d=3$ et $F$ de cardinal au moins $8$,  ceci a \'et\'e \'etabli par Swinnerton-Dyer \cite{SD} et Koll\'ar \cite{K}.
Un corps fini satisfait la propri\'et\'e $C_{1}$.

Soit $k$ un corps $C_{1}$ de caract\'eristique z\'ero, et soit $X$
une $k$-vari\'et\'e projective, lisse, g\'eom\'etriquement connexe.
Supposons $X$ 
g\'eom\'etriquement rationnellement connexe (voir \cite{Klivre}).
 Cette propri\'et\'e est par exemple satisfaite 
si  $X$ est une hypersurface  de degr\'e $d \leq n$
dans  $\P^n_{k}$, ou si $X$ est g\'eom\'etriquement rationnelle, par exemple si
c'est une surface fibr\'ee en coniques sur la droite projective, ou si c'est une surface
de del Pezzo.

Sans pr\'ecision suppl\'ementaire sur $X$ ou sur $k$,
les deux questions  suivantes sont ouvertes :

(1) (Serge Lang) L'ensemble $X(k)$ des points $k$-rationnels de $X$ est-il non vide  ?

(2) \cite[Question 10.11]{CTCIME} Y a-t-il au plus une classe de $R$-\'equivalence sur $X(k)$   ?

Ces questions ont  \'et\'e plus particuli\`erement \'etudi\'ees pour $k=\C(B)$ le corps
des fonctions d'une courbe complexe et $k=\C((t))$ le corps des s\'eries formelles en une variable.
Dans ces deux cas, la question (1) a une r\'eponse affirmative. Pour $k=\C(B)$,
c'est le th\'eor\`eme de Graber, Harris et Starr \cite{GHS}. Pour $k=\C((t))$, on sait \'etablir le r\'esultat
comme cons\'equence dudit th\'eor\`eme \cite[Thm. 7.5]{CTCIME}.

Pour $k=\C(B)$ et $k=\C((t))$, A. Pirutka \cite{P} a donn\'e une r\'eponse affirmative
\`a la question (2) lorsque l'on suppose de plus $X$ rationnellement simplement connexe.

D\'ej\`a en dimension 2,  donc pour les surfaces g\'eom\'etriquement rationnelles,
et m\^{e}me pour les corps $k=\C(B)$ et $k=\C((t))$, 
la question (2) est ouverte.

Soit $k$ un corps infini. Soit $X$ une $k$-surface projective, lisse,
g\'eom\'etriquement connexe, munie d'un $k$-morphisme $f : X \to \P^1_{k}$
dont la fibre g\'en\'erique est une conique lisse. Si $k$ est un corps $C_{1}$,
toutes les fibres lisses poss\`edent des $k$-points. 
 Si   $X(k)/R$ est r\'eduit \`a un \'el\'ement, ou m\^eme simplement  est de cardinal
 plus petit que le corps $k$,
alors il existe un $k$-morphisme $g: \P^1_{k} \to X$
qui compos\'e 	avec $f : X \to \P^1_{k}$ est un $k$-morphisme dominant $h :  \P^1_{k}  \to \P^1_{k}$.  
Le produit fibr\'e $Y$ de $f : X \to \P^1_{k}$ et $h : \P^1_{k}  \to \P^1_{k}$
est une $k$-surface fibr\'ee en coniques   $Y \to \P^1_{k}$
qui admet une section : c'est donc une $k$-surface $k$-rationnelle, et elle domine $X$.
On conclut que $X$ est $k$-unirationnelle.
Soit $Z$ une vari\'et\'e complexe de dimension 3 fibr\'ee en coniques
au-dessus du plan projectif $\P^2_{\C}$. Soit $k=\C(t)$. En consid\'erant un pinceau de droites
dans $\P^2_{\C}$, on voit que le corps des fonctions de $Z$
s'identifie au corps des fonctions d'une $k$-surface projective et lisse $X$  fibr\'ee en coniques sur $\P^1_{k}$.
Si $X(k)/R$ est d\'enombrable, l'argument pr\'ec\'edent montre que $X$ est $k$-unirationnelle. Comme $k=\C(t)$,
on conclut que la vari\'et\'e $Z$ est $\C$-unirationnelle.
  Or on ne s'attend pas \`a ce que  toute vari\'et\'e complexe fibr\'ee en coniques sur $\P^2_{\C}$ 
 soit  $\C$-unirationnelle.   On ne s'attend donc pas \`a 
  une r\'eponse affirmative \`a la question (2) en g\'en\'eral.

Dans un r\'ecent article, Zhiyu Tian  \cite{ZT} \'etablit $X(k)/R = 1$ pour les surfaces
de del Pezzo de degr\'e 4 sur le corps $k=\C((t))$.
Dans cette note, j'observe qu'un th\'eor\`eme de Salberger et Skorobogatov \cite{SalSk}
sur les surfaces de del Pezzo de degr\'e 4 sur un  corps parfait
permet d'\'etablir cet \'enonc\'e pour  beaucoup de corps $C_{1}$,
 en particulier pour $k=C(B)$ le corps des fonctions d'une courbe sur
un corps $C$ alg\'ebriquement clos de caract\'eristique z\'ero.

\begin{theo}
  Soit $k$ un corps $C_{1}$  parfait,  infini, de caract\'eristique diff\'erente de 2.
  Soit $X$ une $k$-surface projective et lisse de l'un des types suivants :
  
  (a) Surface de del Pezzo de degr\'e 4, c'est-\`a-dire intersection lisse
  de deux quadriques dans $\P^4_{k}$.
  
  (b) Surface cubique lisse  dans $\P^3_{k}$ poss\'edant une droite d\'efinie sur $k$.
  
  (c) Surface $X$ munie d'une fibration en coniques $X \to \P^1_{k}$
  poss\'edant exactement 5 fibres g\'eom\'etriques r\'eductibles.
  
  Alors tous les $k$-points de $X$ sont $R$-\'equivalents.
\end{theo}

Sur tout corps $k$, on sait   que les classes de surfaces
consid\'er\'ees en (b) et (c) sont identiques.
 On sait
que la contraction d'une $k$-droite sur une surface cubique lisse
est une surface de del Pezzo de degr\'e 4, et que l'\'eclatement
sur une surface de del Pezzo de degr\'e 4
d'un $k$-point non
situ\'e sur une courbe exceptionnelle est une surface cubique lisse
poss\'edant une $k$-droite. 
Pour ces \'enonc\'es classiques, on  consultera \cite{I} et \cite{M}.
On sait aussi (B. Segre, Manin \cite[Theorem 29.4 and 30.1]{M}, \cite[Prop. 2.3]{CTSaSD}, Koll\'ar \cite{K}, Pieropan \cite{Pie}) que les $k$-surfaces consid\'er\'ees  sont $k$-unirationnelles d\`es qu'elles  contiennent un $k$-point.

Sur $k$ un corps   $C_{1}$
toutes les $k$-vari\'et\'es consid\'er\'ees dans l'\'enonc\'e ont 
des $k$-points, et si $k$ est infini  ces $k$-points sont denses pour la topologie
de Zariski. Une surface de type (a) est donc $k$-birationnelle
\`a une surface de type (b), et inversement.
L'ensemble $X(k)/R$ est invariant par \'eclatement  d'un $k$-point.
Il suffira donc d'\'etablir l'\'enonc\'e pour les surfaces de del Pezzo de degr\'e 4.

 \medskip
 
La trivialit\'e de $X(k)/R$ pour $k$ un corps $C_{1}$
de caract\'eristique z\'ero
\'etait connue pour les surfaces rationnelles fibr\'ees en coniques
avec au plus 4 fibres g\'eom\'etriques r\'eductibles  \cite{CTSk},
donc aussi pour les surfaces $X$  de del Pezzo de degr\'e 4
avec ${\rm Pic}(X)$ de rang au moins 2.

La trivialit\'e de $X(k)/R$ pour $k$ un corps $C_{1}$
n'est pas connue pour une surface cubique lisse $X$
$k$-minimale, i.e. avec ${\rm Pic}(X) $ de rang 1.
Sur $k=\C((t))$, c'est \'etabli pour certaines surfaces cubiques dans 
 \cite{ZT}.

\section{D\'emonstration du th\'eor\`eme}

Soit $k$ un corps de caract\'eristique diff\'erente de 2.

\begin{lem}\label{ouvertquadrique}
 Soit $Q \subset \P^n_{k}$, $n \geq 2$ une quadrique
lisse. Soit $U \subset Q$ un ouvert de Zariski. On a $U(k)/R \leq 1$.
 Plus pr\'ecis\'ement, deux $k$-points $P$ et
 $Q$  de $U$ sont directement $R$-li\'es sur $U$.
\end{lem}
\begin{proof}
\cite[Lemma 3.22 (ii)]{CTSaSD}.
\end{proof}

\begin{lem}\label{ouvertfamillequadriques}
Soit   
$U \subset \P^m_{k} $ un ouvert  non vide. Soit $n \geq 4$. Soit 
$X \subset \P^n_{U}$ une famille lisse de quadriques.
Si $k$ est un corps $C_{1}$,
pour tout ouvert Zariski  non vide $W \subset X$, on 
 a $W(k)/R=1$. Plus pr\'ecis\'ement, deux $k$-points $P$ et
 $Q$ de $W$  sont directement $R$-li\'es sur $W$.
  \end{lem}

\begin{proof}
Notons $p : X \to U$ la projection naturelle.
Quitte \`a remplacer $U$ par l'ouvert $p(W)$, on peut supposer
$U=p(W)$. Soient $P,Q \in W(k)$, et soient $A=p(P)$ et $B=p(Q)$.
Si $A=B$, alors $P$ et $Q$ sont deux $k$-points de la quadrique lisse
$X_{A} \subset \P^n_{k}$, et $P$ et $Q$ sont directement $R$-li\'es
dans l'ouvert $W \cap X_{A} \subset X_{A}$ d'apr\`es le lemme
\ref{ouvertquadrique}, donc aussi sur $W$. Supposons $A \neq B$.
Soit $L \subset  \P^m_{k} $ la droite joignant $A$ et $B$.
La restriction 
 de $p : X \to U$ au-dessus de $L\cap U$
est une famille lisse de quadriques $Y \to L\cap U$ qui poss\`ede une section
rationnelle car le corps $k(L)$ est un corps $C_{2}$ et
toute forme quadratique en au moins 5 variables sur un tel corps
a un point rationnel.
Comme $L\cap U$ est r\'egulier de dimension 1, toute
telle section est morphique.
La fibre g\'en\'erique $Y_{\eta}$ de la famille $Y \to L\cap U$ est une quadrique lisse
qui poss\`ede un point rationnel sur $k(\eta)=k(L)$, donc est
$k(L)$-birationnelle \`a un espace projectif sur le corps $k(L)$.
On sait que la propri\'et\'e d'approximation faible est invariante
par de tels isomorphismes. On peut donc trouver un
$k(L)$-point de $Y_{\eta}$ qui se sp\'ecialise en $P$ au-dessus de $A$
et en $Q$ au-dessus de~$B$.  Le $k$-morphisme
$\theta  : L\cap U \to Y \subset X$ ainsi d\'efini est une $R$-\'equivalence directe entre $P$ et $Q$.
Comme $P$ et $Q$ appartiennent \`a $W$, le $k$-morphisme $\theta^{-1}(W)\subset L$ est
non vide et
$\theta^{-1}(W) \to W$ est une $R$-\'equivalence directe entre $P$ et $Q$ dans $W$.
\end{proof}

\begin{lem}\label{ouverttorseur}
Soient $X$ une $k$-vari\'et\'e int\`egre, $S$ un $k$-groupe r\'eductif connexe,
et $p : {\mathcal T} \to X$ un torseur sur $X$ sous $S$.
Soit $W \subset \mathcal T$ un ouvert,
et soit $U=p(W) \subset X$ l'ouvert qui est son image par $p$.
Si $k$ est un corps infini de dimension cohomologique au plus 1,
l'application $W(k) \to U(k)$ est surjective.
\end{lem}

\begin{proof}
Soit $P \in U(k)$. La fibre $U_{P}$ de $W \to U$ est
un ouvert  non vide de la $k$-vari\'et\'e $ {\mathcal T}_{P}$,
qui est un espace principal homog\`ene sous le $k$-groupe r\'eductif connexe $S$.
Comme la dimension cohomologique de $k$ est  au plus 1,
cet espace principal homog\`ene est trivial, donc la $k$-vari\'et\'e 
$ {\mathcal T}_{P}$ est $k$-isomorphe \`a $S$. 
Tout $k$-groupe r\'eductif connexe sur un corps infini est
$k$-unirationnel (Chevalley, Rosenlicht, Grothendieck \cite[Thm. 18.2]{B}), 
les  $k$-points sont donc
Zariski denses. Ainsi $U_{P}(k) \neq \emptyset$.
\end{proof}

Rappelons le th\'eor\`eme \cite[Thm. 3.8]{SalSk}.

\begin{theo}[Salberger et Skorobogatov]\label{SalSko}
Soit $k$ un corps parfait infini de caract\'eristique diff\'erente de 2.
Soit $X$ une surface de del Pezzo de degr\'e 4
et $ p : {\mathcal T} \to X$ un torseur universel sur $X$
tel que ${\mathcal T}(k) \neq \emptyset$.
Alors il existe  un ouvert  $\Omega$ du plan projectif $\P^2_{k}$ et
 une $k$-vari\'et\'e g\'eom\'etriquement int\`egre $Y$
avec les propri\'et\'es suivantes :

(i) La $k$-vari\'et\'e $Y \times_{k} \P^3_{k}$ est $k$-birationnelle \`a $ {\mathcal T}  \times_{k}\P^1_{k}$.

(ii) La $k$-vari\'et\'e $Y$ est un ouvert dans une famille lisse de quadriques dans
$\P^5_{\Omega}$.
 \end{theo}

\begin{theo}\label{principal}
 Soit $k$ un corps $C_{1}$, parfait,  infini, de caract\'eristique diff\'erente de 2.
 Soit $X$ une
  $k$-surface   de del Pezzo de degr\'e 4.  L'ensemble
$X(k)/R$ est r\'eduit \`a un \'el\'ement.
\end{theo}
\begin{proof}
 Comme le corps $k$ est $C_{1}$ et que $X$ est d\'efinie par l'annulation simultan\'ee de deux formes quadratiques
 en 5 variables, on a $X(k) \neq \emptyset$.
Soit  $p : {\mathcal T} \to X$ le torseur universel  \cite[(2.0.4)]{CTSa} 
de fibre triviale en un $k$-point de $X$.
 Ce torseur a une \'evaluation triviale
en tout $k$-point de $X$. En effet le corps $k$ \'etant $C_{1}$ est de dimension
cohomologique au plus $1$, et donc
 $H^1(k,S)=0$ pour tout $k$-tore $S$.
L'application 
induite ${\mathcal T}(k) \to X(k)$ est donc  surjective  \cite[Lemme 2.7.1]{CTSa}.

Gardons les notations du th\'eor\`eme  \ref{SalSko}.
Notons  $U \subset {\mathcal T} \times_{k}
\P^1_{k}$ et $V \subset Y \times_{k} \P^3_{k}$
des $k$-ouverts non vides isomorphes.

Soit $p_{1} :  U \to X$ le morphisme lisse compos\'e de l'inclusion
$U \subset {\mathcal T} \times_{k}
\P^1_{k}$, de la projection   ${\mathcal T} \times_{k}
\P^1_{k} \to  {\mathcal T}$ et
de $ {\mathcal T} \to X$.  
Soit $X_{0} \subset X$ l'ouvert image de $p_{1}$.
Il r\'esulte du lemme \ref{ouverttorseur} que l'application induite
$U(k) \to X_{0}(k)$ est surjective.

Soient $P,Q$ deux $k$-points de $V$.
Soit $V_{1} \subset Y$ l'ouvert image de $V$ par la projection
$Y \times_{k} \P^3_{k} \to Y$, et soient $A$ et $B$ dans $V_{1}(k)$
les $k$-points images
de $P$ et $Q$. 
D'apr\`es le lemme \ref{ouvertfamillequadriques},
les points $A$ et $B$ sont directement $R$-li\'es sur $V_{1}$.
Il existe donc un $k$-morphisme $D \to V_{1}$
d'un ouvert $D \subset \P^1_{k}$ tel que $A$ et $B$
soient dans l'image de $D(k)$. La restriction de $V \to V_{1}$
\`a $D$ est une famille d'ouverts de $\P^3$.
L'argument d'approximation faible au-dessus du corps
$k(D)$ utilis\'e dans la d\'emonstration du lemme \ref{ouvertfamillequadriques}
\'etablit l'existence d'un $k$-morphisme $D' \to V$, avec $D'$ ouvert de $D$
relevant $D \to V_{1}$, tel que $P$ et $Q$ soient dans l'image de $D'(k)$.
Ainsi $P$ et $Q$ sont directement $R$-li\'es sur $V$.

Combin\'e avec la surjectivit\'e de $U(k) \to X_{0}(k)$,
cela \'etablit que deux $k$-points quelconques de l'ouvert non vide $X_{0}$
sont directement $R$-li\'es. 

Pour \'etablir $X(k)/R=1$, on utilise   la
proposition 3.24 de
\cite[\S 3, Appendice B]{CTSaSD}, qui garantit que l'application
$X_{0}(k) \to X(k)/R$ est surjective.
Cette proposition est \'enonc\'ee sur un corps de caract\'eristique
z\'ero.  Dans le cas consid\'er\'e ici d'une $k$-surface de del Pezzo
de degr\'e 4, $k$-unirationnelle, on v\'erifie que l'argument donn\'e
l\`a, qui utilise la $k$-unirationalit\'e  de $X$,  mentionn\'ee dans l'introduction,  et  \cite[\S 3, Appendice B,  Prop. 3.23]{CTSaSD},
vaut sur tout corps $k$  infini de caract\'eristique diff\'erente de 2.
 \end{proof}

 \begin{rema}  
   Dans la d\'emonstration du th\'eor\`eme \ref{principal}, l'hypoth\`ese que le corps $k$ est  parfait 
 vient de cette m\^eme hypoth\`ese dans le  th\'eor\`eme \ref{SalSko}.  
   \end{rema}

\end{document}